\documentclass[12pt, reqno]{amsart}
\usepackage{amsmath, amsthm, amscd, amsfonts, amssymb, graphicx, color}
\usepackage[bookmarksnumbered, colorlinks, plainpages]{hyperref}
\hypersetup{colorlinks=true,linkcolor=red, anchorcolor=green, citecolor=cyan, urlcolor=red, filecolor=magenta, pdftoolbar=true}

\newtheorem{theorem}{Theorem}[section]

\newtheorem{proposition}[theorem]{Proposition}

\theoremstyle{definition}
\newtheorem{definition}[theorem]{Definition}

\theoremstyle{remark}
\newtheorem{remark}[theorem]{Remark}
\numberwithin{equation}{section}

\begin{document}
\setcounter{page}{1}

\title[Convex hull-like property]{Convex hull-like property and \\ supported images of open sets}

\author[B. Ricceri]{B. Ricceri}

\address{$^1$ Department of Mathematics, University of Catania, Viale A. Doria 6, 95125 Catania, Italy}
\email{\textcolor[rgb]{0.00,0.00,0.84}{ricceri@dmi.unict.it}}

\dedicatory{Dedicated to Professor Anthony To-Ming Lau, with esteem and friendship}



\begin{abstract}
In this note, as a particular case of a more general result, we obtain the following theorem:

 Let $\Omega\subseteq {\bf R}^n$ be a
non-empty bounded open set and let $f:\overline {\Omega}\to {\bf R}^n$ be a  continuous function which is $C^1$ in $\Omega$.
Then, at least one of the following assertions holds:
\begin{itemize}
\item[$(a)$] $f(\Omega)\subseteq \hbox {\rm conv}(f(\partial \Omega))\ .$
\item[$(b)$] There exists a non-empty open set $X\subseteq \Omega$, with $\overline {X}\subseteq \Omega$, satisfying the following property:
for every continuous  function $g:\Omega\to {\bf R}^n$ which is $C^1$ in $X$, there exists $\tilde\lambda\geq 0$ such that, for each
$\lambda>\tilde\lambda$, the Jacobian determinant of the function $g+\lambda f$ vanishes at some point of $X$.
\end{itemize}

\smallskip

As a consequence, if $n=2$ and $h:\Omega\to {\bf R}$ is a non-negative function, for each $u\in C^2(\Omega)\cap C^1(\overline {\Omega})$
satisfying in $\Omega$ the Monge-Amp\`ere equation
$$u_{xx}u_{yy}-u_{xy}^2=h\ ,$$
one has
$$\nabla u(\Omega)\subseteq \hbox{\rm conv}(\nabla u(\partial\Omega))\ .$$
\end{abstract} \maketitle

\section{Introduction and preliminaries}
Here and in what follows, $\Omega$ is a non-empty relatively compact and open set in a topological space $E$, with $\partial \Omega\neq \emptyset$, and
$Y$ is a real locally convex Hausdorff topological vector space. $\overline {\Omega}$ and $\partial \Omega$ denote the closure and the
boundary of $\Omega$, respectively. Since $\overline {\Omega}$ is compact, $\partial\Omega$, being closed, is compact too.

\smallskip
Let us first recall some well-known definitions.

\smallskip
Let $S$ be a subset of $Y$ and let $y_0\in S$. As usual, we say that $S$ is supported at $y_0$
 if there exists $\varphi\in Y^*\setminus \{0\}$
such that $\varphi(y_0)\leq \varphi(y)$ for all $y\in S$. If this happens, of course
$y_0\in\partial S$.

\smallskip
Further, extending a maximum principle definition for real-valued functions,
a continuous function $f:\overline {\Omega}\to Y$ is said to  satisfy the convex hull property in $\overline {\Omega}$ (see \cite{DKS, K} and references therein)  if
$$f(\Omega)\subseteq \overline {\hbox {\rm conv}}(f(\partial \Omega))\ ,$$
$\overline {\hbox {\rm conv}}(f(\partial \Omega))$ being the closed convex hull of $f(\partial \Omega)$.

\smallskip
When dim$(Y)<\infty$, since $f(\partial\Omega)$ is compact, conv$(f(\partial\Omega))$ is compact too and so $\overline
{\hbox {\rm conv}}(f(\partial \Omega))=\hbox {\rm conv}(f(\partial\Omega))$.

\smallskip
 A function $\psi:Y\to {\bf R}$ is said to be quasi-convex if, for each $r\in {\bf R}$, the set $\psi^{-1}(]-\infty,r])$ is convex.
 
\smallskip
Notice the following proposition:
\begin{proposition}\label{PROPOSITION 1}
For each pair $A, B$ of non-empty subsets of $Y$, the following assertions are equivalent:
\begin{itemize}
\item[$(a_1)$] $A\subseteq \overline{\hbox {\rm conv}}(B)\ .$
\item[$(a_2)$] For every continuous and quasi-convex function $\psi:Y\to {\bf R}$, one has
$$\sup_A\psi\leq \sup_B\psi\ .$$
\end{itemize}
\end{proposition}
\begin{proof}
Let $(a_1)$ hold. Fix any continuous and quasi-convex function $\psi:Y\to {\bf R}$. Fix $\tilde y\in A$.
Then, there is a net $\{y_{\alpha}\}$ in conv$(B)$ converging to $\tilde y$. So, for each $\alpha$, we have
$y_{\alpha}=\sum_{i=1}^k\lambda_i z_i$, where $z_i\in B$, $\lambda_i\in [0,1]$ and $\sum_{i=1}^k\lambda_i=1$.
By quasi-convexity, we have
$$\psi(y_{\alpha})=\psi\left ( \sum_{i=1}^k\lambda_iz_i\right )\leq \max_{1\leq i\leq k}\psi(z_i)\leq \sup_B\psi$$
and so, by continuity,
$$\psi(\tilde y)=\lim_{\alpha}\psi(y_{\alpha})\leq\sup_B\psi$$
which yields $(a_2)$.

Now, let $(a_2)$ hold. Let $x_0\in A$. If $x_0\not\in \overline{\hbox {\rm conv}}(B)$, by the standard separation theorem, there
would be $\psi\in Y^*\setminus\{0\}$ such that $\sup_{\overline{\hbox {\rm conv}}(B)}\psi<\psi(x_0)$, against $(a_2)$. So,
$(a_1)$ holds.
\end{proof}
\medskip
Clearly, applying Proposition \ref{PROPOSITION 1}, we obtain the following one:

\medskip
\begin{proposition}\label{PROPOSITION 2}
 For any continuous function $f:\overline {\Omega}\to Y$, the following assertions are equivalent:
\begin{itemize}
\item[$(b_1)$] $f$ satisfies the convex hull property in $\overline {\Omega}$.
\item[$(b_2)$] For every continuous and quasi-convex function $\psi:Y\to {\bf R}$, one has
$$\sup_{x\in \Omega}\psi(f(x))=\sup_{x\in\partial \Omega}\psi(f(x)).$$
\end{itemize}
\end{proposition}
\medskip
In view of Proposition \ref{PROPOSITION 2}, we now introduce the notion of convex hull-like property for functions defined in $\Omega$ only.

\medskip
\begin{definition}\label{DEFINITION 1} A continuous function $f:\Omega\to Y$ is said to satisfy the convex hull-like property in $\Omega$ if, for every continuous and quasi-convex function
$\psi:Y\to {\bf R}$, there exists $x^*\in\partial \Omega$ such that
$$\limsup_{x\to x^*}\psi(f(x))=\sup_{x\in \Omega}\psi(f(x))\ .$$
\end{definition}
\medskip
We have

\medskip
\begin{proposition}\label{PROPOSITION 3}
Let $g:\overline {\Omega}\to Y$ be a continuous function and let $f=g_{|\Omega}$.

Then, the following assertions are equivalent:
\begin{itemize}
\item[$(c_1)$] $f$ satisfies the convex hull-like property in $\Omega$.
\item[$(c_2)$] $g$ satisfies the convex hull property in $\overline {\Omega}$.
\end{itemize}
\end{proposition}
\medskip
\begin{proof}Let $(c_1)$ hold. Let $\psi:Y\to {\bf R}$ be any continuous and quasi-convex function.
Then, by Definition \ref{DEFINITION 1}, there exists $x^*\in\partial \Omega$ such that
$$\limsup_{x\to x^*}\psi(f(x))=\sup_{x\in \Omega}\psi(f(x))\ .$$
But
$$\limsup_{x\to x^*}\psi(f(x))=\psi(g(x^*))$$
and hence
$$\sup_{x\in\partial \Omega}\psi(g(x))=\sup_{x\in \Omega}\psi(g(x))\ .$$
So, by Proposition \ref{PROPOSITION 2}, $(c_2)$ holds.

Now, let $(c_2)$ hold. Let $\psi:Y\to {\bf R}$ be any continuous and quasi-convex function. Then, by
Proposition \ref{PROPOSITION 2}, one has
$$\sup_{x\in\partial \Omega}\psi(g(x))=\sup_{x\in \Omega}\psi(g(x))\ .$$
Since $\partial \Omega$ is compact and $\psi\circ g$ is continuous,
there exists $x^*\in\partial \Omega$ such that
$$\psi(g(x^*))=\sup_{x\in\partial \Omega}\psi(g(x))\ .$$
But
$$\psi(g(x^*))=\lim_{x\to x^*}\psi(f(x))$$
and, by continuity again,
$$\sup_{x\in \Omega}\psi(g(x))=\sup_{x\in \overline {\Omega}}\psi(g(x))$$
and so
$$\lim_{x\to x^*}\psi(f(x))=\sup_{x\in \Omega}\psi(f(x))$$
which yields $(c_1)$.\end{proof}
\medskip
After the above preliminaries, we can declare the aim of this short note:  to establish Theorem \ref{THEOREM 1} below jointly with some of its consequences.

\medskip
\begin{theorem}\label{THEOREM 1}For any continuous function
 $f:\Omega\to Y$, at least one of the following assertions holds:
 \begin{itemize}
\item[$(i)$] $f$ satisfies the convex hull-like property in $\Omega$\ .
\item[$(ii)$]There exists a non-empty open set $X\subseteq \Omega$, with $\overline {X}\subseteq \Omega$,  satisfying the following property:
for every continuous function $g:\Omega\to Y$, there exists $\tilde\lambda\geq 0$
such that, for each $\lambda>\tilde\lambda$, the set $(g+\lambda f)(X)$
is supported at one of its points.
\end{itemize}
\end{theorem}

\section{ Proof of Theorem \ref{THEOREM 1}}
\bigskip
Assume that $(i)$ does not hold. So, we are assuming that there exists
a continuous and quasi-convex function $\psi:Y\to {\bf R}$ such that
\begin{equation}\label{1}
\limsup_{x\to z}\psi(f(x))<\sup_{x\in \Omega}\psi(f(x))
\end{equation}
for all $z\in \partial \Omega$.

In view of $(\ref{1})$, for each $z\in \partial \Omega$, there exists
an open neighbourhood $U_z$ of $z$ such that
$$\sup_{x\in U_z\cap \Omega}\psi(f(x))<\sup_{x\in \Omega}\psi(f(x))\ .$$
Since $\partial\Omega$ is compact, there are finitely many $z_1,...,z_k\in
\partial \Omega$ such that
\begin{equation}\label{2}\partial \Omega\subseteq \bigcup_{i=1}^kU_{z_i}\ .
\end{equation}
Put
$$U=\bigcup_{i=1}^kU_{z_i}\ .$$
Hence
$$\sup_{x\in U\cap \Omega}\psi(f(x))=\max_{1\leq i\leq k}\sup_{x\in U_{z_i}\cap \Omega}\psi(f(x))<
\sup_{x\in \Omega}\psi(f(x))\ .$$
Now, fix a number $r$ so that
\begin{equation}\label{3}\sup_{x\in U\cap \Omega}\psi(f(x))<r<\sup_{x\in \Omega}\psi(f(x)) 
\end{equation}
and set
$$K=\{x\in \Omega : \psi(f(x))\geq r\}\ .$$
Since $f, \psi$ are continuous, $K$ is closed in $\Omega$. But,
since $K\cap U=\emptyset$ and $U$ is open, in view of $(\ref{2})$,
$K$ is closed in $E$. Hence, $K$ is compact since  $\overline {\Omega}$ is so.
By $(\ref{3})$, we can fix $\bar x\in \Omega$ such that $\psi(f(\bar x))>r$. Notice that the set
$\psi^{-1}(]-\infty,r])$ is closed and convex. So, thanks to
the standard separation theorem,  there exists
a non-zero continuous linear functional $\varphi:Y\to {\bf R}$ such that
\begin{equation}\label{4}\varphi(f(\bar x))<\inf_{y\in \psi^{-1}(]-\infty,r])}\varphi(y)\ .
\end{equation}
Then, from $(\ref{4})$, it follows
$$\varphi(f(\bar x))<\inf_{x\in \Omega\setminus K}\varphi(f(x))\ .$$
Now, choose $\rho$ so that
$$\varphi(f(\bar x))<\rho<\inf_{x\in \Omega\setminus K}\varphi(f(x))$$
and set
$$X=\{x\in \Omega : \varphi(f(x))<\rho\}\ .$$
Clearly, $X$ is a non-empty open set contained in $K$. Now, let $g:\Omega\to Y$ be any continuous function.
Set
$$\tilde\lambda=\inf_{x\in X}{{\varphi(g(x))-\inf_{z\in K}\varphi(g(z))}\over {\rho-\varphi(f(x))}}\ .$$
Fix $\lambda>\tilde\lambda$. So, there is $x_0\in X$ such that
$${{\varphi(g(x_0))-\inf_{z\in K}\varphi(g(z))}\over {\rho-\varphi(f(x_0))}}<\lambda\ .$$
  From this, we get
  \begin{equation}\label{5}\varphi(g(x_0))+\lambda\varphi(f(x_0))<\lambda \rho+\inf_{z\in K}\varphi(g(z))\ .
  \end{equation}
By continuity and compactness, there
exists $\hat x\in K$ such that
\begin{equation}\label{6}\varphi(g(\hat x)+\lambda f(\hat x))\leq\varphi(g(x))+\lambda f(x))
\end{equation}
for all $x\in K$.
Let us prove that $\hat x\in X$.
Arguing by contradiction, assume that $\varphi(f(\hat x))\geq \rho$.
Then, taking $(\ref{5})$ into account, we would have
$$\varphi(g(x_0))+\lambda\varphi(f(x_0))<\lambda\varphi(f(\hat x))+\varphi(g(\hat x))$$
contradicting $(6)$. So, it is true that $\hat x\in X$, and, by $(\ref{6})$, the set
$(g+\lambda f)(X)$ is supported at its point $g(\hat x)+\lambda f(\hat x)$.

\bigskip
\section{ Applications}

\bigskip
The first application of Theorem \ref{THEOREM 1} shows a strongly bifurcating behaviour of certain equations in ${\bf R}^n$.

\medskip
\begin{theorem}\label{THEOREM 2}
Let $\Omega$ be a non-empty bounded open subset of ${\bf R}^n$ and let $f:\Omega\to {\bf R}^n$ a
continuous function.

Then,  at least one of the following assertions holds:
\begin{itemize}
\item[$(d_1)$] $f$ satisfies the convex hull-like property in $\Omega$.
\item[$(d_2)$] There exists a non-empty open set $X\subseteq \Omega$, with $\overline {X}\subseteq \Omega$, satisfying the following property:
for every continuous function $g:\Omega\to {\bf R}^n$, there exists $\tilde\lambda\geq 0$
such that, for each $\lambda>\tilde\lambda$, there exist $\hat x\in X$ and two
sequences $\{y_k\}$, $\{z_k\}$ in ${\bf R}^n$, with
$$\lim_{k\to \infty}y_k=\lim_{k\to \infty}z_k=g(\hat x)+\lambda f(\hat x)\ ,$$
such that, for each $k\in {\bf N}$, one has
\begin{itemize}
\item[$(j)$] the equation
$$g(x)+\lambda f(x)=y_k$$
has no solution in $X$\ ;
\item[$(jj)$] the equation
$$g(x)+\lambda f(x)=z_k$$
has two distinct solutions $u_k, v_k$ in $X$ such that
$$\lim_{k\to \infty}u_k=\lim_{k\to \infty}v_k=\hat x\ .$$
\end{itemize}
\end{itemize}
\end{theorem}
\smallskip
\begin{proof}Apply Theorem \ref{THEOREM 1} with $E=Y={\bf R}^n$. Assume that $(d_1)$ does not hold.  Let $X\subseteq \Omega$ be an open
set as in $(ii)$ of Theorem \ref{THEOREM 1}. Fix any continuous function $g:\Omega\to {\bf R}^n$.
Then, there is
some $\tilde\lambda\geq 0$ such that, for each $\lambda>\tilde\lambda$, there exists $\hat x\in X$
such that the set $(g+\lambda f)(X)$ is supported at $g(\hat x)+\lambda f(\hat x)$. As we observed at the
beginning, this implies that
$g(\hat x)+\lambda f(\hat x)$ lies in the boundary of $(g+\lambda f)(X)$.
Therefore, we can find a sequence $\{y_k\}$ in ${\bf R}^n\setminus (g+\lambda f)(X)$ converging
to $g(\hat x)+\lambda f(\hat x)$. So, such a sequence satisfies $(j)$. For each $k\in {\bf N}$,
denote by $B_k$ the open ball of radius ${{1}\over {k}}$ centered at $\hat x$. Let $k$ be such that
$B_k\subseteq X$. The set $(g+\lambda f)(B_k)$ is not open since  its boundary contains the point
$g(\hat x)+\lambda f(\hat x)$. Consequently, by the invariance of domain theorem (\cite{Z}, p. 705),
the function $g+\lambda f$ is not injective in $B_k$. So, there are $u_k, v_k\in B_k$, with $u_k\neq v_k$
such that
$$g(u_k)+\lambda f(u_k)=g(v_k)+\lambda f(v_k)\ .$$
Hence, if we take
$$z_k=g(u_k)+\lambda f(u_k)\ ,$$
the sequences $\{u_k\}, \{v_k\}, \{z_k\}$ satisfy $(jj)$ and the proof is complete.\end{proof}

\medskip
\begin{remark}\label{REMARK 1}
Notice that, in general, Theorem \ref{THEOREM 2} is no longer true when $f:\Omega\to {\bf R}^m$ with $m>n$. In this connection,
consider the case $n=1$, $m=2$, $\Omega=]0,\pi[$ and $f(\theta)=(\cos\theta,\sin\theta)$ for $\theta\in [0,\pi]$. So,
for each $\lambda>0$, on the one hand, the function $\lambda f$ is injective, while, on the other hand, $\lambda f(]0,\pi[)$
is not contained in conv$(\{f(0),f(\pi)\})$.
\end{remark}

\medskip
If $S\subseteq {\bf R}^n$ is a non-empty open set, $x\in S$ and $h:S\to {\bf R}^n$ is a $C^1$ function, we denote
by det$(J_h(x))$ the Jacobian determinant of $h$ at $x$.

\smallskip
Another important consequence of Theorem \ref{THEOREM 1} is as follows:

\medskip
\begin{theorem}\label{THEOREM 3} Let $\Omega$ be a non-empty bounded open subset of ${\bf R}^n$ and let
 $f:\Omega\to {\bf R}^n$ be a $C^1$ function. 
 
Then,  at least one of the following assertions holds:
\begin{itemize}
\item[$(e_1)$] $f$ satisfies the convex hull-like property in $\Omega$\ .
\item[$(e_2)$] There exists a non-empty open set $X\subseteq \Omega$, with $\overline {X}\subseteq \Omega$, satisfying the following property:
for every continuous  function $g:\Omega\to {\bf R}^n$ which is $C^1$ in $X$, there exists $\tilde\lambda\geq 0$ such that, for each
$\lambda>\tilde\lambda$,
one has $$\hbox {\rm det}(J_{g+\lambda f}(\hat x))=0$$ for some $\hat x\in X$.
\end{itemize}
\end{theorem}
\smallskip
\begin{proof}  Assume that $(e_1)$ does not hold. Let $X$ be an open set as in $(ii)$ of Theorem \ref{THEOREM 1}. Let $g:\Omega\to {\bf R}^n$ be a continuous function which
is $C^1$ in $X$. Then, there is
some $\tilde\lambda\geq 0$ such that, for each $\lambda>\tilde\lambda$, there exists $\hat x\in X$
such that the set $(g+\lambda f)(X)$ is supported at $g(\hat x)+\lambda f(\hat x)$.  By remarks already made, we infer that
the function $g+\lambda f$ is not a local homeomorphsim at $\hat x$, and so $\hbox {\rm det}(J_{g+\lambda f}(\hat x))=0$ in view of
the classical inverse function theorem.\end{proof}
\medskip
In turn, here is a consequence of Theorem \ref{THEOREM 3} when $n=2$.

\medskip
\begin{theorem}\label{THEOREM 4}
Let $\Omega$ be a non-empty bounded open set of ${\bf R}^2$, let
$h:\Omega\to {\bf R}$ be a continuous function and let $\alpha, \beta:\Omega\to {\bf R}$ be two $C^1$ functions such that
$|\alpha_x\beta_y-\alpha_y\beta_x|+|h|>0$ and $(\alpha_x\beta_y-\alpha_y\beta_x)h\geq 0$ in $\Omega$.

Then, any $C ^1$ solution $(u,v)$ in $\Omega$ of the system
\begin{equation}\label{7}
\left\{
  \begin{array}{ll}
    u_xv_y-u_yv_x=h,  \\
\beta_y u_x-\beta_x u_y-\alpha_yv_x+\alpha_xv_y =0
    \end{array}
\right.
\end{equation}
 satisfies the convex hull-like property in $\Omega$.
 \end{theorem}
\smallskip
\begin{proof}Arguing by contradiction, assume that $(u,v)$ does not satisfy the convex hull-like property in $\Omega$. Then, by Theorem \ref{THEOREM 3}, applied taking
$f=(u,v)$ and $g=(\alpha,\beta)$, there exist $\lambda>0$ and $(\hat x,\hat y)\in \Omega$ such that
$$\hbox {\rm det}(J_{g+\lambda f}(\hat x,\hat y))=0\ .$$
On the other hand, for each $(x,y)\in \Omega$, we have
\begin{eqnarray*}\hbox {\rm det}(J_{g+\lambda f}(x, y))&=&(u_xv_y-u_yv_x)(x,y)\lambda^2+(\beta_y u_x-\beta_x u_y-\alpha_yv_x+\alpha_xv_y)(x,y)\lambda\\&+&
(\alpha_x\beta_y-\alpha_y\beta_x)(x,y)
\end{eqnarray*}
and hence
$$h(\hat x,\hat y)\lambda^2+(\alpha_x\beta_y-\alpha_y\beta_x)(\hat x,\hat y)=0$$
which is impossible in view of our assumptions.\end{proof}

\medskip
We conclude by highlighting two applications of Theorem \ref{THEOREM 4}.

\medskip
\begin{theorem}\label{THEOREM 5}Let $\Omega$ be a non-empty bounded open subset of ${\bf R}^2$, let $h:\Omega\to {\bf R}$ be a continuous non-negative function and
let $w\in C^2(\Omega)$ be a function satisfying in $\Omega$ the  Monge-Amp\`ere equation
$$w_{xx}w_{yy}-w_{xy}^2=h\ .$$
Then, the gradient of $w$ satisfies the convex hull-like property in $\Omega$.
\end{theorem}
\smallskip
\begin{proof}It is enough to observe that $(w_x,w_y)$ is a $C^1$ solution in $\Omega$ of the system $(\ref{7})$ with
$\alpha(x,y)=-y$ and $\beta(x,y)=x$ and that such $\alpha, \beta$ satisfy the assumptions of Theorem \ref{THEOREM 4}.\end{proof}

\medskip
\begin{theorem}\label{THEOREM 6} Let $\Omega$ be a non-empty bounded open subset of ${\bf R}^2$ and let $\beta:\Omega\to
{\bf R}$ be a $C^1$ function. Assume that there exists another $C^1$ function $\alpha:\Omega\to {\bf R}$ so that
the function $\alpha_x\beta_y-\alpha_y\beta_x$ vanishes at no point of $\Omega$. 

 Then, for any function $u\in C^1(\Omega)\cap C^0(\overline {\Omega})$ satisfying in $\Omega$ the equation
$$\beta_yu_x-\beta_xu_y=0\ ,$$
 one has
$$\sup_{\Omega}u=\sup_{\partial\Omega}u$$
and
$$\inf_{\Omega}u=\inf_{\partial\Omega}u\ .$$
\end{theorem}
\smallskip
\begin{proof} Observe that the function $(u,0)$ satisfies the system $(\ref{7})$ with $h=0$ and that the assumptions
of Theorem \ref{THEOREM 4} are fulfilled. So, $(u,0)$ satisfies the convex hull-like property in $\Omega$. Since $u\in C^0(\overline {\Omega})$,
the conclusion follows from Proposition \ref{PROPOSITION 3}.\end{proof}
\medskip
{\bf Acknowledgement.} The author has been supported by the Gruppo Nazionale per l'Analisi Matematica, la Probabilit\`a e le loro Applicazioni (GNAMPA) of the Istituto Nazionale di Alta Matematica (INdAM).

\bibliographystyle{amsplain}

\end{document}